\documentclass[11pt, reqno, A4]{amsart}

\usepackage{amssymb,latexsym,amsmath,amsfonts}
\usepackage{mathrsfs}
\usepackage{enumerate}

\numberwithin{equation}{section}
\theoremstyle{definition}
\newtheorem{definition}{Definition}[section]

\theoremstyle{remark}
\newtheorem{remark}[definition]{Remark}
\theoremstyle{plain}
\newtheorem{theorem}[definition]{Theorem}
\newtheorem{lemma}[definition]{Lemma}
\newtheorem{proposition}[definition]{Proposition}

\newtheorem{result}[definition]{Result}

\newcommand{\eps}{\varepsilon}

\newcommand{\zt}{\zeta}

\newcommand{\zahl}{\mathbb{Z}}

\newcommand{\B}{\boldsymbol{{\sf B}}}
\newcommand{\id}{\mathbb{I}}

\newcommand{\col}{\boldsymbol{\sf a}}
\newcommand{\mlti}{\mathscr{I}}
\newcommand{\zer}{\boldsymbol{\sf Z}}
\newcommand\bsf[1]{\boldsymbol{{\sf {#1}}}}
\newcommand{\esym}{\mathscr{S}}

\newcommand{\bdy}{\partial\mathbb{D}}
\newcommand{\OM}{\Omega}
\newcommand{\Dee}{\mathbb{D}}
\newcommand{\cDee}{\overline{\mathbb{D}}}

\newcommand\quotn[1]{\mathbb{E}_{{#1}}}
\newcommand\sympd[1]{\mathbb{G}_{{#1}}}
\newcommand{\diff}{\mathscr{N}}
\newcommand{\bol}{\mathbb{B}}

\newcommand{\smoo}{\mathcal{C}}
\newcommand{\hol}{\mathcal{O}}

\newcommand{\excep}{\mathfrak{S}_{1,\,n}}
\newcommand{\gene}{{\widetilde{\OM_1}}_{,\,n}}
\newcommand{\geep}{\mathscr{G}_{1,\,n}}
\newcommand\de[1]{{#1}^\bullet}
\newcommand\cyc[1]{\boldsymbol{c}_{{#1}}}
\newcommand\cmpn[1]{\boldsymbol{{\sf C}}_{{#1}}}
\newcommand\interR[2]{[{#1}\,.\,.\,{#2}]}
\newcommand{\vP}{\varPsi_n}
\newcommand{\dg}{{\rm diag}}
\newcommand{\sgn}{{\rm sgn}}
\newcommand{\II}{I^{{\sf C}}}
\newcommand{\kro}{\boldsymbol{\delta}}
\newcommand\mltix[1]{{}^{{#1}}\!\!\mathscr{I}}
\newcommand{\res}{\mathscr{R}_n}

\newcommand{\bcdot}{\boldsymbol{\cdot}}
\newcommand{\ess}{\widetilde{s}}
\newcommand{\exx}{\widetilde{x}}
\newcommand{\aye}{\widetilde{y}}
\newcommand{\pee}{{}^{\raisebox{-2pt}{$\scriptstyle{n}$}}{p}}
\newcommand{\npee}{\widetilde{{}^{\raisebox{-1pt}{$\scriptstyle{n}$}}{p}}}

\newcommand{\impl}{\Longrightarrow}

\newcommand{\lrarw}{\longrightarrow}

\newcommand{\Cn}{\mathbb{C}^n}
\newcommand{\cplx}{\mathbb{C}}
\newcommand{\rea}{\mathbb{R}}

\begin{document}

\title[Domains associated with $\mu$-synthesis]{A family of domains \\ 
associated with $\mu$-synthesis}
\author{Gautam Bharali}
\address{Department of Mathematics, Indian Institute of Science, Bangalore -- 560 012}
\email{bharali@math.iisc.ernet.in}
\thanks{This work is supported in part by a Centre for Advanced Study grant.}
\keywords{Categorical quotients, interpolation, Nevanlinna--Pick problem, $\mu$-synthesis}
\subjclass[2010]{Primary: 30E05, 32F45; Secondary: 47A56, 93D21}

\begin{abstract} 
We introduce a family of domains\,---\,which we call the $\mu_{1,\,n}$-quotients\,--- associated
with an aspect of $\mu$-synthesis. We show that the natural association that the symmetrized polydisc
has with the corresponding spectral unit ball is also exhibited by the $\mu_{1,\,n}$-quotient and its
associated unit ``$\mu_E$-ball''. Here, $\mu_E$ is the structured singular value for the
case $E = \{[w]\oplus(z\id_{n-1})\in \cplx^{n\times n}: z,w\in \cplx\}$, $n = 2, 3, 4,\dots$ Specifically:
we show that, for such an $E$, the Nevanlinna--Pick interpolation problem with matricial data in a
unit ``$\mu_E$-ball'', and in general position in a precise sense, is equivalent to a
Nevanlinna--Pick interpolation problem for the associated $\mu_{1,\,n}$-quotient. Along the
way, we present some characterizations for the $\mu_{1,\,n}$-quotients.   
 
\end{abstract}
\maketitle

\section{Introduction and Main Results}\label{S:intro}

This article is devoted to studying the following infinite family of domains ($\Dee$ here will
denote the open unit disc with centre $0\in \cplx$):
\begin{multline*}
 \quotn{n}\,:=\,\left\{(x,y)\in \Cn\times \cplx^{n-1} :\,\text{the zero set of} \ 
 \left(1+\sum\nolimits^{n-1}_{j=1}(-1)^j y_j z^j\right.\right. \\
 \left.\left.- w\left(\sum\nolimits^{n-1}_{j=0}(-1)^j x_{j+1}z^j\right)\right) \ 
 \text{does not intersect $\overline{\Dee}^2$}\right\}, \; n\geq 2,
\end{multline*}
which we shall call the {\em $\mu_{1,\,n}$-quotients}.
These domains are closely associated with an aspect of $\mu$-synthesis.  
We will provide a couple
of characterizations for $\quotn{n}$ that make it easier to work with these domains
(that each $\quotn{n}$ is a domain is a classical argument; we defer this matter to Remark~\ref{Rem:DOM} below). 
The focus of this work, however, is to establish the connection between the $\quotn{n}$'s
and (the relevant aspect of) $\mu$-synthesis.
\smallskip

$\mu$-synthesis is a part of the theory of robust control of systems comprising
interconnected electronic or mechanical devices each of whose outputs depend linearly
on the inputs. Various performance measures are given by appropriate
$\rea_+$-homogeneous functionals on the space of matrices associated with such
systems\,---\,see, for instance, \cite{helton:n-Efae82}. The ``$\mu$'' in
$\mu$-synthesis refers to such a class of cost functions. Fix $n\in \zahl_+$, $n\geq 2$,
and let $E$ be a linear subspace of $\cplx^{n\times n}$. The functional
\[
 \mu_E(A)\,:=\,\left(\,\inf\{\|X\| : X\in E \ \text{and $(\id - AX)$ is singular}\}\right)^{-1}, \; \;
 A\in \cplx^{n\times n},
\]
is called a {\em structured singular value}. Here, $\|\bcdot\|$ denotes the operator norm 
relative to the Euclidean norm on $\Cn$. Typically, the subspace $E$ consists of all complex
$n\times n$ matrices having a fixed block-diagonal structure. If $E = \cplx^{n\times n}$,
then $\mu_E = \|\bcdot\|$, while if $E$ is the space of all scalar matrices, then 
$\mu_E$ is the spectral radius. The motivation for, and the definition of, $\mu_E$ comes from
the theory of efficient stabilization of systems in which the uncertainties in their governing
parameters are highly structured: the subspace $E$ is meant to encode the structure of
the perturbations to such systems.
\smallskip

In much the same way that a necessary condition for desigining a controller that stabilizes
the aforementioned system (with unstructured uncertainties) is the existence of an interpolant
for certain Nevanlinna--Pick data with values in the unit $\|\bcdot\|$-ball\,---\,see
\cite[Chapter~4]{francis:cH^{infty}ct87}, for instance\,---\,with {\em structured} uncertainties one
needs to understand the Nevanlinna--Pick interpolation problem for the unit ``$\mu_E$-ball'' for a given $E$.
\smallskip

At this juncture, we shift our focus entirely
to the Nevanlinna--Pick interpolation problem. We refer readers (who aren't
already familiar) to the pioneering work of John Doyle \cite{doyle:afssu82} for the control-theory
motivations behind $\mu_E$. With $E$ as above,
let $\OM_E := \{W\in \cplx^{n\times n} : \mu_E(W) < 1\}$. The Nevanlinna--Pick
interpolation problem for $\OM_E$ is the following:
\smallskip

\begin{itemize}
\item[$(*)$] {\em Given $M$ distinct points
$\zt_1,\dots,\zt_M\in \Dee$ and matrices
$W_1,\dots, W_M$ in $\OM_E$, find necessary conditions and sufficient conditions
on $\{(\zt_1, W_1),\dots,(\zt_M, W_M)\}$ for the existence
of a holomorphic map $F:\Dee\longrightarrow \OM_E$ satisfying $F(\zt_j)=W_j, \
j=1,\dots,M$.}
\end{itemize}
\smallskip

\noindent{When $E$ is the class of all scalar matrices in $\cplx^{n\times n}$, $\OM_E$
is the so-called spectral unit ball, which we denote by $\OM_n$.
The problem $(*)$ has been studied intensively for $\OM_n$.
Bercovici {\em et al.} \cite{bercFoiasTann:sclt91} have given a characterization 
for the interpolation data $\{(\zt_1,W_1),\dots,(\zt_M,W_M)\}$ to admit an $\OM_n$-valued
interpolant. However, this characterization involves a non-trivial search over a region in $\cplx^{n^2M}$. Thus,
there is interest in finding alternative characterizations that would at least reduce the dimension
of the search-region: see, for instance, \cite{aglerYoung:2b2sNPp04, bercovici:svcNPid203}.
This was one of the motivations behind the ideas in the paper \cite{aglerYoung:cldC2si99}
by Agler\,\&\,Young, wherein they introduced the symmetrized bidisc. Its $n$-dimensional analogue
(the symmetrized polydisc, denoted by $\sympd{n}$) was introduced by Costara in \cite{costara:osNPp05}.
The importance  of $\sympd{n}$ to $\mu$-synthesis is as follows:
\begin{itemize}
 \item[$a)$] $\dim(\OM_n)\gg \dim(\sympd{n})$, yet, whenever the matrices
  $W_1, W_2,\dots, W_M\in \OM_n$ lie off an explicitly defined set 
  $\mathfrak{S}_n\varsubsetneq \OM_n$, which is of zero Lebesgue measure, the
  problem $(*)$ is equivalent to an associated Nevanlinna--Pick problem for $\sympd{n}$. 
\end{itemize}

\noindent{(Also see \cite{nikolovPflugThomas:sNPCFpng3} for an improvement of $(a)$ when $n = 2, 3$.)}
\smallskip

For most of the systems alluded to above, the associated $E$ comprises matrices whose diagonal
blocks are either scalar matrices or rank-one matrices. We address here the next level of complexity
in the block structure of $E$. The domains $\quotn{n}$, $n\geq 2$,
introduced above are the analogues of the symmetrized polydiscs $\sympd{n}$ when (for a fixed
$n\geq 2$)
\begin{equation}\label{E:struct}
 E\,\equiv\,E^{1,\,n}\,:=\,\{[w]\oplus (z\id_{n-1}) : z, w\in \cplx\}
\end{equation}
(here, $\id_{n-1}$ denotes the $(n-1)\times(n-1)$ identity matrix). Theorem~\ref{T:equiv} below
is precisely the statement $(a)$ with the domains $\OM_E$, for the above choice of $E$, replacing
$\OM_n$. For this choice of $E$, we shall denote $\OM_E$ as $\OM_{1,\,n}$.
\smallskip

The feature $(a)$ is not the only useful insight that $\sympd{n}$ brings to the Nevanlinna--Pick
problem on $\OM_n$. The set $\mathfrak{S}_n$
(which we have not defined; but see \cite[Theorem~2.1]{costara:osNPp05}) helps explain
certain subtleties of the interpolation
problem. We shall elaborate upon these after stating Theorem~\ref{T:equiv}, but we mention here
that the preceding remark motivates our explicit description of the set
$\gene$\,---\,the analogue of $(\OM_n\setminus\mathfrak{S}_n)$ for $\OM_{1,\,n}$\,---\,in
Theorem~\ref{T:equiv}. It is also important to mention that a special case of our domains $\quotn{n}$
is the tetrablock. It was introduced by Abouhajar {\em et al.} \cite{abouWhYoung:sldrmu-s07} and
is the domain $\quotn{2}$.
\smallskip 

To describe $\gene$, we shall need the following:

\begin{definition}\label{D:nonderog} A matrix $A\in \cplx^{n\times n}$ is said to be
{\em non-derogatory} if $A$ admits a cyclic vector. Therefore, $A$ being non-derogatory is
equivalent to $A$ being similar to the companion matrix of its characteristic polynomial\,---\,i.e.,
if $z^n+\sum_{j=1}^ns_jz^{n-j}$ denotes the characteristic polynomial, then
\[
A \ \text{is non-derogatory} \ \iff \ \text{$A$ is similar to}
                 \begin{bmatrix}
                        \ 0  & {} & {} & -s_n \ \\
                        \ 1  & 0  & \text{\LARGE{$\boldsymbol{0}$}}   & -s_{n-1} \ \\
                        \ {} & \ddots  & \ddots & \vdots \ \\
                                \ \text{\LARGE{$\boldsymbol{0}$}} & {} & 1 & -s_{1} \
                \end{bmatrix}_{n\times n}.
\]
\end{definition}

We shall make use of some notations throughout this work. For a matrix $A\in \cplx^{n\times n}$,
$\de{A}$ will denote the $(n-1)\times (n-1)$ matrix obtained by deleting the first row and column
of $A$. For any pair of integers $m\leq n$,
$\interR{m}{n}$ will denote the integer subset $\{m, m+1,\dots,n\}$.  
Assume that $n\in \zahl_+$ is
fixed; for any $j$ such that $1\leq j\leq n$, $\mlti^j$ will denote the set of all increasing
$j$-tuples in $\interR{1}{n}^j$. Finally, for $n$ and $j$ as described, for $I\in \mlti^j$,
and for any $A\in \cplx^{n\times n}$, $A_I$ will denote the $j\times j$ submatrix of
$A$ whose rows and columns are indexed by $I$. Having defined
these notations, we can state our first result:

\begin{theorem}\label{T:equiv}
Let $n\geq 2$, write any $A\in \cplx^{n\times n}$ as $A = [a_{j,\,k}]$,
and let $\OM_{1,\,n}$ be as defined above. Define:
\begin{align*}
  \gene\,:=\,\big\{A\in \OM_{1,\,n} :\,&\,\text{$\de{A}$ is non-derogatory, and} \\
  &\,\text{$(a_{2,1},\dots,a_{n,1})$ is a cyclic vector of $\de{A}$}\big\}.
\end{align*}
Then:
\begin{enumerate}[1)]
  \item $\OM_{1,\,n}\setminus\gene$\,$(=:\excep)$ has zero Lebesgue measure.
  \item Define the map $\pi_n : \cplx^{n\times n}\lrarw \cplx^{2n-1}$ by
  \begin{multline*}
   \pi_n(A)\,:= \\ 
   \bigg(a_{1,1},
   \sum_{I\in \mlti^2\,:\,i_1=1}\!\!\det(A_I),\dots,\!\!\sum_{I\in \mlti^n\,:\,i_1=1}\!\!\det(A_I);
   \sum_{I\in \mlti^1\,:\,i_1\geq 2}\!\!\det(A_I),\dots,
   \!\!\sum_{I\in \mlti^{n-1}\,:\,i_1\geq 2}\!\!\det(A_I)\bigg).
  \end{multline*}
  $\pi_n$ is holomorphic and maps $\OM_{1,\,n}$ onto $\quotn{n}$.
  \item Let $\zt_1,\dots,\zt_M$ be distinct points in $\Dee$ and let
  $W_1,\dots, W_M$ belong to $\gene$. Then,
  there exists a holomorphic map
  $F: \Dee\lrarw \OM_{1,\,n}$ satisfying $F(\zt_j) = W_j$ for every $j\leq M$
  if and only if there exists a holomorphic map $f: \Dee\lrarw \quotn{n}$ satisfying
  $f(\zt_j) = \pi_n(W_j)$ for every $j\leq M$. 
\end{enumerate}
\end{theorem}

Engineers have had some success in numerically computing solutions to the problem $(*)$.
These methods are based on iterative schemes that are supported by convincing, but largely
{\em heuristic}, arguments. However, we now know that the problem $(*)$ is ill-conditioned in a specific
sense. The set $\mathfrak{S}_n$ in $(a)$ (and its analogue
$\excep$, given by Theorem~\ref{T:equiv}) gives us a precise description of this problem:
\begin{itemize}
 \item[$b)$] (following \cite[Example~2.3]{aglerYoung:2psNPp00} by Agler--Young) There exist
 continuous one-parameter families of Nevanlinna--Pick data
 $\{(\zt_1, W_{1,\alpha}), (\zt_2, W_{2,\alpha})\}_{\alpha\in \Dee}$  with
 $(W_{1,\alpha}, W_{2,\alpha})\in (\OM_n\setminus\mathfrak{S}_n)^2 \ \forall
 \alpha\neq 0$ such that there exist $\OM_n$-valued interpolants
 $\forall \alpha\neq 0$, but {\bf none} for $\alpha = 0$. In this case,
 either $W_{1,0}\in \mathfrak{S}_n$ or $W_{2,0}\in \mathfrak{S}_n$.
\end{itemize}

\noindent{This provides useful information for testing the stability of some of
the numerical algorithms used. It is the information that $(b)$ provides that is our
second motivation for constructing analogues of $\sympd{n}$ for the case of
$E^{1,\,n}$.}
\smallskip

Indeed, Abouhajar {\em et al.} have shown \cite[Remark~9.5-$(iv)$]{abouWhYoung:sldrmu-s07}
that the problem $(*)$ for  $\OM_{1,2}$ is also ill-conditioned,
exactly as described in $(b)$ with $\OM_{1,2}$ replacing $\OM_n$ therein. This pathology
extends to $\OM_{1,\,n}$ for all $n\geq 2$. It turns out
that, analogous to $(b)$, the problem lies in either $W_{1,0}$ or $W_{2,0}$ belonging
to $\excep$ (as defined in Theorem~\ref{T:equiv}-(1)). In fact, it is
\cite[Remark~9.5-$(iv)$]{abouWhYoung:sldrmu-s07} that led us to intuit what
$\excep$ must be for general $n$. 
\smallskip
  
Our second main result provides a necessary condition for the existence of an interpolant
that solves the problem $(*)$ for 
$E^{1,\,n}$. For this, we must give some definitions. For each $(x,y)\in \Cn\times \cplx^{n-1}$,
let us define:
\begin{align*}
 P_n(z; x)\,&:=\,\sum\nolimits^{n-1}_{j=0}(-1)^j x_{j+1}z^j, \\
 Q_n(z; y)\,&:=\,1+\sum\nolimits^{n-1}_{j=1}(-1)^j y_j z^j, \ \ z\in \cplx, \\
 \vP(z; x,y)\,&:=\,\begin{cases}
 					\dfrac{P_n(z; x)}{Q_n(z;y)} \\
 					\text{\small{(with the understanding that, in evaluating $\vP(\bcdot\,; x,y)$, any}} \\
 					\text{\small{common linear factors of $P_n(\bcdot\,; x)$ and $Q_n(\bcdot\,; y)$ are
 					first cancelled)}}.
 					\end{cases}
\end{align*}
With these definitions, we can state our next result.

\begin{theorem}\label{T:nec}
Let $\zt_1,\dots,\zt_M$ be distinct points in $\Dee$ and let
$W_1,\dots, W_M$ in $\OM_{1,\,n}$, $n\geq 2$. Express the map
$\pi_n$ as $\pi_n = (X,Y) : \cplx^{n\times n}\lrarw \Cn\times \cplx^{n-1}$.
If there exists a holomorphic map
$F: \Dee\lrarw \OM_{1,\,n}$ satisfying $F(\zt_j) = W_j$ for every $j$, then,
for each $z\in \cDee$, the matrix
\[
 M_z\,:=\,\left[\frac{1-\overline{\vP(z; X(W_j), Y(W_j)})\,\vP(z; X(W_k), Y(W_k))}
              {1-\overline{\zt_j}\zt_k}\right]_{j,k=1}^M
\]
is positive semi-definite.
\end{theorem}

\begin{remark}\label{Rem:impli}
Implicit in the statement of Theorem~\ref{T:nec} is that if $(x, y)\in \quotn{n}$, then
the rational function $\vP(\bcdot\,; x,y)$ has no poles in $\Dee$. In fact, much more
can be said about $\vP(\bcdot\,; x,y)$, as we shall see in Section~\ref{S:char}.
\end{remark}

Theorem~\ref{T:nec} is an easy corollary to a certain characterization of the set
$\quotn{n}$ in terms of the functions $\vP(\bcdot\,; x,y)$. It also turns out
that the sets $\quotn{n}$, $n\geq 2$, form a certain hierarchy in the sense that membership
in $\quotn{n+1}$ can be characterized in terms of membership in
$\quotn{n}$, $n\geq 2$. The precise results (Theorems~\ref{T:char1} and \ref{T:char2})
will be presented in Section~\ref{S:char}.
\smallskip

We ought to state that the theorems presented in this section address only a {\em small part}
of what control engineers need. The chief utility to engineers is that, in view of $(b)$ above
and the paragraph that follows it, the set $(\OM_{1,\,n}\setminus\gene)$ raises a very
specific flag in testing numerical methods for constructing Nevanlinna--Pick interpolants
that rely on limit processes. The question arises: given that, in real-world stabilization
problems (with structured uncertainties) one encounters other forms of the space $E$,
what can one say about Theorems~\ref{T:equiv} and \ref{T:nec} for general $E$\,? We make some
remarks on this issue, and on the subject of {\em categorical quotients}\,---\,of which the
reader gets a {\em very} fleeting glimpse in Section~\ref{S:prelim}\,---\,in
Section~\ref{S:main} (Remarks~\ref{Rem:good-q} and \ref{Rem:returns}) below.
\medskip

\section{A Few Preliminary Lemmas}\label{S:prelim}

This section is devoted to a few lemmas that we will need in the subsequent sections.
\smallskip

In the following lemma, we shall follow the notation introduced in Section~\ref{S:intro}
and the standard multi-index notation. A diagonal $n\times n$ matrix having the number
$a_j$ as the entry in its $j$th row and column will be denoted by
$\dg(a_1,\dots, a_n)$. 

\begin{lemma}\label{L:det}
Fix an integer $n\geq 2$ and let $A\in \cplx^{n\times n}$. Then:
\begin{equation}\label{E:key-det}
 \det\left(\id_n - A\,\dg(z_1,\dots, z_n)\right)\,=\,1+\sum_{j=1}^n(-1)^j\sum_{I\in \mlti^j}\det(A_I)z^I.
\end{equation}
\end{lemma}
\begin{proof}
Let us denote the matrix on the left-hand side above by $B$. As usual, we
write $A = [a_{j,\,k}]$ and $S_n$ for the group of permutations of $n$ distinct
objects. We write down the classical expansion of $\det(B^T)$  to see that
\begin{align}
 \det(B)\,&=\,\sum_{\sigma\in S_n}\sgn(\sigma)\prod_{k=1}^n
 \left(\kro_{\sigma(k),\,k}-a_{\sigma(k),\,k}\,z_k\right) \notag \\
 &=\,1+\sum_{\sigma\in S_n}\sgn(\sigma)
 \sum_{j=1}^n(-1)^j\left[\sum_{I\in \mlti^j}\left(\prod\nolimits_{s\in I}\!a_{\sigma(s),\,s}\right)
 \left(\prod\nolimits_{t\in \II}\!\kro_{\sigma(t),\,t}\right)z^I\right], \label{E:expan}
\end{align}
where $\II$ is the abbreviation for $\interR{1}{n}\setminus I$, and with the understanding
that a product indexed by the null set equals $1$.
Clearly, the second product on the right-hand side of \eqref{E:expan} is non-zero if and only
if $\sigma$ fixes the subset $\II$. For any subset $J\subseteq \interR{1}{n}$, write
\[
 {\sf Fix}(J)\,:=\,\{\sigma\in S_n : \sigma \ \text{fixes $J$}\}.
\]
Then, from \eqref{E:expan}, we get
\[
 \det(B)\,=\,1+\sum_{j=1}^n(-1)^j\sum_{I\in \mlti^j}\left(
 \sum\nolimits_{\sigma\in {\sf Fix}(\II)}\!\sgn(\sigma)\prod\nolimits_{s\in I}\!a_{\sigma(s),\,s}\right)z^I.
\]
Given the definition of the submatrices $A_I$, the above identity is precisely \eqref{E:key-det}.
\end{proof}

For the next lemma, we present a convention that we will follow in this article.
The notation $\cplx^*\oplus GL_{n-1}(\cplx)$, $n\geq 2$, will denote the group
(with respect to matrix multiplication) of $n\times n$ matrices $G$ that are block-diagonal,
with the $(1,1)$-entry of $G$ being a non-zero complex number and
$\de{G}\in GL_{n-1}(\cplx)$.

\begin{lemma}\label{L:orbit}
Let $(x, y)\in \Cn\times\cplx^{n-1}$, $n\geq 2$, and let $\pi_n : \cplx^{n\times n}\lrarw
\Cn\times\cplx^{n-1}$ be the map defined in Theorem~\ref{T:equiv}. Let $A\in
\pi_n^{-1}\{(x, y)\}$. Then, the congugacy orbit
\[
 \{G^{-1}AG : G\in \cplx^*\oplus GL_{n-1}(\cplx)\}\,\subseteq\,\pi_n^{-1}\{(x, y)\}.
\]
\end{lemma}
\begin{proof}
We shall denote the conjugacy orbit $\{G^{-1}AG : G\in \cplx^*\oplus GL_{n-1}(\cplx)\}$ as 
$O_A$. It suffices to show that $\pi_n$ is constant on $O_A$. As in Section~\ref{S:intro},
we write $\pi_n = (X, Y)$. We will denote any element $G\in \cplx^*\oplus GL_{n-1}(\cplx)$
as $g\oplus \Gamma$:  $g$ being the $(1,1)$-entry of $G$, and $\de{G} = \Gamma$.
It is a classical fact that, by definition:
\begin{equation}\label{E:Y}
 Y(G^{-1}AG)\,=\,\big(\esym_{n-1,1}(\sigma(\Gamma^{-1}\de{A}\Gamma)),\dots,
 \esym_{n-1,\,n-1}(\sigma(\Gamma^{-1}\de{A}\Gamma))\big) \ \  
\end{equation}
where
\begin{align*}
 \esym_{n-1,\,j}\,&:=\,\text{the $j$-th elementary symmetric polynomial in $n-1$ indeterminates}, \\
 \sigma(B)\,&:=\,\text{the list of eigenvalues of the matrix $B$, listed according to multiplicity.}
\end{align*}
As $\esym_{n-1,\,j}$ is a similarity invariant, \eqref{E:Y} implies that $Y$ is constant on $O_A$.
\smallskip

Therefore, it suffices to show that $X$ is constant on $O_A$.
For any $G = g\oplus\Gamma$ as above, and $j = 2,\dots, n$, we compute:
\begin{align}
 \sum_{I\in \mlti^j}\!\det(A_I)\,&=\,\sum_{I\in \mlti^j}\!\det((G^{-1}AG)_I) \notag \\
 &=\,\sum_{I\in \mlti^j\,:\,i_1=1}\!\!\det((G^{-1}AG)_I) +
 		\sum_{I\in \mlti^j\,:\,i_1\geq 2}\!\!\det((G^{-1}AG)_I)  \notag \\
 &=\,\sum_{I\in \mlti^j\,:\,i_1=1}\!\!\det((G^{-1}AG)_I) + 
 		\esym_{n-1,\,j}(\sigma(\Gamma^{-1}\de{A}\Gamma)). \label{E:cnjgtn}
\end{align}
The left-hand side of \eqref{E:cnjgtn} is a constant. Therefore, it follows from
\eqref{E:cnjgtn} that the $j$-th component of the map $X : \cplx^{n\times n}\lrarw \Cn$,
$j = 2,\dots n$, is constant on $O_A$. And, of course, the $(1,1)$ entry
of $(g\oplus\Gamma)^{-1}A(g\oplus\Gamma)$ does not vary with $G$. Hence the lemma.
\end{proof}
 
The next two lemmas will be essential to the proof of Theorem~\ref{T:equiv}.

\begin{lemma}\label{L:imp}
Let $(x, y)\in \Cn\times\cplx^{n-1}$, $n\geq 2$. There exist polynomials
$p_k\in \cplx[x, y]$, $k=1,\dots, (n-1)$, such that, if we define
\[
 B(x, y)\,:=\,\begin{bmatrix}
 			\ x_1     &\vline & p_1(x,y)     & \dots & p_{n-2}(x,y) & p_{n-1}(x,y) \ \\ \hline
 			\ 1        &\vline & 0               & {}     & {}               & (-1)^ny_{n-1} \ \\
 			\ 0        &\vline & 1  & 0        & \text{\LARGE{$\boldsymbol{0}$}} & (-1)^{n-1}y_{n-2} \ \\
 			\ \vdots &\vline & {} & \ddots & \ddots & \vdots \ \\
 			\ 0        &\vline & \text{\LARGE{$\boldsymbol{0}$}} & {}  & 1 &  y_{1} \
                \end{bmatrix},
\]
then, for each $j=2,\dots, n$,
\[
 \sum_{I\in \mlti^j\,:\,i_1=1}\!\!\det(B(x,y)_I)\,=\,x_j.
\]
Furthermore, for a given $(x,y)$, $p_1(x,y),\dots p_{n-1}(x,y)$ are the {\em unique} numbers
for which the above equations hold true.
\end{lemma}
\begin{proof}
Let $\B$ be the $n\times n$ matrix obtained by replacing the entries $p_k(x, y)$ by
the unknowns $Z_k$, $k =1,\dots, (n-1)$, in the matrix $B(x, y)$ given above. We shall
need some auxiliary objects. First, given a vector $w\in \Cn$, for each integer
$m\in \interR{1}{n}$, let us define the matrices
\[
 M(m; w,y)\,:=\,\begin{bmatrix}
 			    \ w_{n+1-m} &\vline & w_{n+2-m} & \dots & w_{n-1} & w_{n} \ \\ \hline
 			    \ 1              &\vline & 0               & {}     & {}        & (-1)^my_{m-1} \ \\
 			    \ 0              &\vline & 1               & 0        
 			                                                   & \text{\LARGE{$\boldsymbol{0}$}} & (-1)^{m-1}y_{m-2} \ \\
                       \ \vdots &\vline & {} & \ddots & \ddots & \vdots \ \\
                       \ 0        &\vline & \text{\LARGE{$\boldsymbol{0}$}} & {}  & 1 &  y_{1} \
                       \end{bmatrix}_{m\times m}.
\]
For $m$ as above, we shall write:
\begin{align*}
 \mltix{m}^j\,&:=\,\text{the set of all increasing $j$-tuples in $\interR{1}{m}$}, \ 1\leq j\leq m, \\
 \mltix{m}^j(1)\,&:=\,\{I\in \mltix{m}^j : i_1=1\}, \ 1\leq j\leq m, \\
 \mltix{m}^j(1,2)\,&:=\,\{I\in \mltix{m}^j : i_1=1, \ i_2=2\}, \ 2\leq j\leq m.
\end{align*}
Finally, we shall define, for $m$ as above, and $1\leq k\leq m$,
\[
 \Phi(k,m; w,y)\,:=\,\sum_{I\in \mltix{m}^k(1)}\!\det(M(m; w,y)_I).
\]
We begin with an elementary observation. Suppose, {\em for the moment}, $n\geq 4$.
Then, for $(m,k)$ such that $3\leq k\leq m-1$, we have
\begin{align}
 \Phi(k,m; w,y)\,&=\,\sum_{I\in \mltix{m}^k(1,2)}\!\!\det(M(m; w,y)_I) + 
 				\sum_{I\in \mltix{m}^k(1)\setminus\mltix{m}^k(1,2)}\!\!\!\det(M(m; w,y)_I) \notag \\
 &=\,-\Phi(k-1, m-1; w,y) + \sum_{I\in \mltix{m}^k(1)\setminus\mltix{m}^k(1,2)}\!\!\!\det(M(m; w,y)_I).
 \label{E:deflate1} 
\end{align}
This follows by expanding each determinant in the first sum with respect to its first column
and from the fact that, as $3\leq k\leq m-1$, the $(1,1)$-cofactor of each submatrix
$M(m; w,y)_I$, $I\in \mltix{m}^k(1,2)$, has at least one zero-column. As for $\Phi(2,m; w,y)$,
it is easy to see, owing to the structure of $\de{M(m; w,y)}$, that
\begin{equation}\label{E:k=2}
 \Phi(2,m; w,y)\,=\,-w_{n+2-m}+y_1w_{n+1-m}.
\end{equation}

It is possible to simplify the second sum in the equation \eqref{E:deflate1} further. We argue
along the lines described just after \eqref{E:deflate1}: we expand each determinant with respect
to its first column. However, there is a difference in this case. The $(1,1)$-cofactor of
each relevant $M(m; w,y)_I$ will have a zero-column {\em except} when 
$I = (1, m-k+2,\dots, m)$. Note that, as $k\leq m-1$, $m-k+2\neq 2$. The
$(1,1)$-cofactor of $M(m; w,y)_{(1, m-k+2,\dots, m)}$ is the companion matrix
of the polynomial $X^{k-1}-y_1X^{k-2}+\dots+(-1)^{k-2}y_{k-2}X+(-1)^{k-1}y_{k-1}$.
Thus:
\[
 \sum_{I\in \mltix{m}^k(1)\setminus\mltix{m}^k(1,2)}\!\!\!\det(M(m; w,y)_I)\,=\,w_{n+1-m}y_{k-1}.
\]
Combining this with \eqref{E:deflate1}, we get
\begin{equation}\label{E:deflate2}
 \Phi(k,m; w,y)\,=\,-\Phi(k-1, m-1; w,y) + w_{n+1-m}y_{k-1}, \ \ 3\leq k\leq m-1.
\end{equation}

The conclusions of the lemma can easily be established for $n=2, 3$ (we leave it to
the reader to check this). We shall establish the lemma for $n\geq 4$. Recall the definition
of the matrix $\B$. Treating $(Z_1,\dots, Z_{n-1})$ as unknowns, the following:
\begin{equation}\label{E:s_leq}
 \sum_{I\in \mlti^j\,:\,i_1=1}\!\!\det(\B_I)\,=\,x_j, \ \ j=2,\dots,n,
\end{equation}
is a system of $(n-1)$ algebraic equations in $(n-1)$ unknowns.
\smallskip

Observe that the matrix $\B$ is the matrix $M(n; w,y)$ with
$w = (x_1, Z_1,\dots, Z_{n-1})$. Thus, taking $w = (x_1, Z_1,\dots, Z_{n-1})$ 
in \eqref{E:k=2} and \eqref{E:deflate2} and applying \eqref{E:deflate2} recursively, we
see that the system \eqref{E:s_leq} is
a lower-triangular system of linear equations in $(Z_1,\dots, Z_{n-1})$. From the recursion relation
\eqref{E:deflate2}, we get that the coefficient of the unknown $Z_j$ in the $j$-th equation of
\eqref{E:s_leq} (which concerns the sum of the $(j+1)$-st principal minors of $\B$) is $(-1)^j$,
$1\leq j\leq (n-2)$. Finally, expanding $\det(\B)$ along the first {\em row}, we see that the
coefficient of $Z_{n-1}$ in the last equation of \eqref{E:s_leq} is $(-1)^{n-1}$. It follows
from Cramer's rule that each $Z_j$ is a polynomial $p_j$ in $(x,y)$. By our definition of $\B$,
these polynomials, $p_1,\dots, p_{n-1}$, are the required polynomials. The uniqueness
statement follows from the fact that, for a fixed $(x,y)$, the system \eqref{E:s_leq} has
a unique solution
\end{proof}

We continue to follow the notation presented just before the statement of Theorem~\ref{T:equiv}.
Further notation: if $S$ is a square matrix, then $\cmpn{S}$ will denote the companion matrix
of its characteristic polynomial (normalized as in Definition~\ref{D:nonderog}).

\begin{lemma}\label{L:uni_det}
Fix an integer $n\geq 2$, and write any $A\in \cplx^{n\times n}$ as $A = [a_{j,\,k}]$. Define
\begin{align*}
 \geep\,:=\,\big\{A\in \cplx^{n\times n} :\,&\,\text{$\de{A}$ is non-derogatory, and} \\
 &\,\text{$(a_{2,1},\dots, a_{n,1})$ is a cyclic vector of $\de{A}$}\big\}.
\end{align*}
Let $A, B\in \geep$. Suppose $A, B\in \pi_n^{-1}\{(x,y)\}$ for some $(x,y)\in \Cn\times\cplx^{n-1}$.
If $\de{A} = \de{B}$ and  $(a_{2,1},\dots, a_{n,1}) = 
(b_{2,1},\dots, b_{n,1})$, then $(a_{1,2},\dots, a_{1,\,n})=(b_{1,2},\dots, b_{1,\,n})$.
\end{lemma}
\begin{remark}\label{Rem:column}
In the proof of the above lemma\,---\,as elsewhere in this
article\,---\,a vector in $\cplx^k$, 
$1\leq k <\infty$, will also be treated (without any change in notation) as a $k\times 1$ complex matrix.
\end{remark}
\begin{proof}
By assumption, $\de{A}$ is non-derogatory. It is well-known that any matrix $G\in GL_{n-1}(\cplx)$
such that $G^{-1}\de{A}G = \cmpn{\de{A}}$ must be of the form.
\[
 G\,=\,[\cyc{} \ \ \de{A}\cyc{}\,\dots\,(\de{A})^{n-2}\cyc{}],
\]
where $\cyc{}$ is some cyclic vector of $\de{A}$. Thus, the matrix
\begin{equation}\label{E:gamma}
 \Gamma\,:=\,[\col \ \ \de{A}\col{}\,\dots\,(\de{A})^{n-2}\col{}],
\end{equation}
where $\col:=[a_{2,1}\dots a_{n,1}]^T$, is the {\em unique} element
in $GL_{n-1}(\cplx)$ with the two properties
\begin{align*}
 \Gamma^{-1}\de{A}\,\Gamma\,&=\,\cmpn{\de{A}}, \\
 \Gamma\,[1 \  0,\,\dots\,0]^T\,&=\,\col.
\end{align*}
We will denote elements $X\in \cplx^*\oplus GL_{n-1}(\cplx)$ using the
abbreviated notation introduced in the proof of Lemma~\ref{L:orbit}. By what we have
just discussed:
\[
 (1\oplus\Gamma)^{-1}A
 (1\oplus\Gamma)\,=\,\begin{bmatrix}
 					\ a_{1,1}	&\vline & \ [a_{1,2} \ \ a_{1,3}\,\dots\,a_{1,\,n}]\Gamma \ \\ \hline
 					\ 1		&\vline & { } \\
 					\ 0		&\vline & { } \\
 					\ \vdots	&\vline & \cmpn{\de{A}} \\
 					\ 0		&\vline & { }
 					\end{bmatrix}.
\]
Call the above matrix $\widehat{A}$. By Lemma~\ref{L:orbit}, $\widehat{A}\in \pi_n^{-1}\{(x,y)\}$.
Thus, by Lemma~\ref{L:imp}, it follows\,---\,compare the matrix above with the matrix
$B(x,y)$ in Lemma~\ref{L:imp}\,---\,that
\[
 [a_{1,2}, \ \ a_{1,3}\,\dots\,a_{1,\,n}]\,=\,[p_1(x,y) \ \ p_2(x,y)\,\dots\,p_{n-1}(x,y)]\Gamma^{-1}.
\]
\vskip1mm

However, the argument above applies to $B$ as well, and as $\de{A} = \de{B}$ and
$(a_{2,1},\dots, a_{n,1}) = (b_{2,1},\dots, b_{n,1})$, the matrix $\Gamma$ given by
\eqref{E:gamma} works for $B$ as well. And as $A, B\in \pi_n^{-1}\{(x,y)\}$, we can conclude that
\begin{align*}
 [a_{1,2}, \ \ a_{1,3}\,\dots\,a_{1,\,n}]\,&=\,[p_1(x,y) \ \ p_2(x,y)\,\dots\,p_{n-1}(x,y)]\Gamma^{-1} \\
 								     &=\,[b_{1,2}, \ \ b_{1,3}\,\dots\,b_{1,\,n}].
\end{align*}
\end{proof}

\section{Two Characterizations of $\quotn{n}$}\label{S:char}

As hinted in Section~\ref{S:intro}, Theorem~\ref{T:nec} follows from a certain
characterization of $\quotn{n}$. This characterization is the focus of this section.
We begin with a proposition that explains the origins of the (somewhat odd-looking)
sets $\quotn{n}$. Readers familiar with
\cite{abouWhYoung:sldrmu-s07} will notice that the following proposition is
a generalization of \cite[Theorem~9.1]{abouWhYoung:sldrmu-s07}.

\begin{proposition}\label{P:mu-to-quot}
A point $(x,y)\in \Cn\times \cplx^{n-1}$ belongs to $\quotn{n}$ if
and only if there exists a matrix $A\in \OM_{1,\,n}$ such that
$\pi_n(A) = (x,y)$. Furthermore, if $(x,y)\in \quotn{n}$, then the matrix $B(x,y)$ defined
in the statement of Lemma~\ref{L:imp} belongs to $\OM_{1,\,n}$. 
\end{proposition}
\begin{remark}
The first part of the above is, essentially, part~(2) of Theorem~\ref{T:equiv}.
\end{remark}
\begin{proof}
Let $E^{1,\,n}$ be as in \eqref{E:struct}. Given $r > 0$ and a matrix $A\in \cplx^{n\times n}$,
$\mu_{E^{1,\,n}}(A)\leq 1/r$ if and only if for, any matrix $M\in E^{1,\,n}$ that satisfies
\[
 \det(\id - AM)\,=\,0,
\]
$\|M\|\geq r$. Let us write $\pi_n$ as $(X,Y): \cplx^{n\times n}\lrarw \Cn\times\cplx^{n-1}$.
It follows from Lemma~\ref{L:det} that if the $M$ above is written as
$M = [w]\oplus(z\id_{n-1})$, then
\begin{equation}\label{E:poly}
  \det(\id - AM)\,=\,\left(1+\sum\nolimits^{n-1}_{j=1}(-1)^j Y_j(A) z^j\right) -
  w\left(\sum\nolimits^{n-1}_{j=0}(-1)^j X_{j+1}(A)z^j\right).
\end{equation}
The preceding discussion is summarized as follows:
\smallskip

\begin{itemize}
 \item[(\textbullet)] $\mu_{E^{1,\,n}}(A)\leq 1/r$, $r > 0$, if and only if the
 zero set of the polynomial on the right-hand side of \eqref{E:poly}
 is disjoint from $(r\Dee)^2$.
\end{itemize}

Now, suppose $A\in \OM_{1,\,n}$. Then there exists an $r_0 > 1$ such that 
$\mu_{E^{1,\,n}}(A)\leq 1/r_0$. It follows from (\textbullet) that the zero
set of the polynomial on the right-hand side of \eqref{E:poly}
is disjoint from $(r_0\Dee)^2$, whence it is disjoint from $\overline{\Dee}^2$.
Thus $(X, Y)(A) = \pi_n(A)\in \quotn{n}$.
\smallskip

Let $(x,y)\in \quotn{n}$. Let $p_1,\dots, p_{n-1}$ be the polynomials
provided by Lemma~\ref{L:imp} and let $A$ be the matrix $B(x,y)$ defined
in Lemma~\ref{L:imp}. Since $\de{A}$ is a companion matrix, it
follows by examination of its last column that $Y(A) = (y_1,\dots, y_{n-1})$. Thus, from
the definition of $\pi_n$ and by Lemma~\ref{L:imp}, we have
\begin{equation}\label{E:key-eq}
  \pi_n(A) =(x,y).
\end{equation}  
As $(x,y)\in \quotn{n}$, it follows that there exists a
small positive constant $\eps_0$ such that the zero set of the polynomial
\[
 \left(1+\sum\nolimits^{n-1}_{j=1}(-1)^j y_j z^j\right)
 - w\left(\sum\nolimits^{n-1}_{j=0}(-1)^j x_{j+1}z^j\right)
\]
is disjoint from $((1+\eps_0)\Dee)^2$. From (\textbullet) and
\eqref{E:key-eq}, we have $\mu_{E^{1,\,n}}(A)\leq 1/(1+\eps_0) < 1$.
This completes the proof.
\end{proof}

The first theorem of this section is a consequence of Proposition~\ref{P:mu-to-quot}. In order to
state it, we need a definition. Fix an integer $n\geq 2$ and let $(x,y)\in \Cn\times \cplx^{n-1}$.
Let $P_n(\bcdot\,;x)$ and $Q_n(\bcdot\,;y)$ be the polynomials defined just prior to
Theorem~\ref{T:nec}, and define
\[
 \res(x,y)\,:=\,{\sf Res}(P_n(\bcdot\,;x), Q_n(\bcdot\,;y)),
\]
where ${\sf Res}$ denotes the resultant of a pair of univariate polynomials.

\begin{theorem}\label{T:char1}
Fix an integer $n\geq 2$, and, for $(x,y)\in \Cn\times \cplx^{n-1}$, let
$\vP(\bcdot\,;x,y)$ be the rational function defined in Section~\ref{S:intro}.
The point $(x,y)\in \quotn{n}$ if and only if the following two
conditions are satisfied:
\begin{itemize}
 \item[(I)] $\left.\vP(\bcdot\,; x,y)\right|_{\cDee} \in \hol(\Dee)\cap \smoo(\cDee)$, and
 \[
  \sup_{z\in \bdy}|\vP(z; x,y)|\,<\,1;
 \]
 \item[(II)] If $\res(x,y) = 0$, then every common zero of $P_n(\bcdot\,;x)$ and
 $Q_n(\bcdot\,;y)$ lies outside $\cDee$.
\end{itemize}
\end{theorem}
\begin{proof}
In this proof, for any polynomial $p\in \cplx[z,w]$, $\zer(p)$ will denote its zero set in
$\cplx^2$. Let us fix $(x,y)\in \Cn\times \cplx^{n-1}$ and write:
\[
 \pee(z,w; x,y)\,:=\,Q_n(z; y) - wP_n(z; x).
\]
We will begin with some basic observations. First:
\begin{align}
 z_0 \ \text{is a common zero of $P_n(\bcdot\,;x)$ and
 $Q_n(\bcdot\,;y)$}\,\iff\,\{z_0\}\times\cplx\subset\,&\zer(\pee(\bcdot\,; x,y)), \label{E:vert-z} \\
 P_n(z_0; x) = 0 \ \text{and} \ Q_n(z_0; y)\neq 0\,\iff\,(\{z_0\}\times\cplx)\,\cap\,&\zer(\pee(\bcdot\,; x,y))
 = \varnothing. \label{E:when-null}
\end{align}
Secondly: in view of \eqref{E:vert-z} and \eqref{E:when-null}, it follows that for any $z_0\in \cplx$:
\begin{multline}\label{E:z-struc}
 (\{z_0\}\times\cplx)\cap\zer(\pee(\bcdot\,; x,y)) \subset
 \{z_0\}\times(\cplx\setminus\cDee) \\
   \impl\,\begin{cases}
 		\{z_0\} \ \text{is not a common zero of $P_n(\bcdot\,;x)$ and $Q_n(\bcdot\,;y)$},  \\
 		Q_n(z_0; y)\neq 0, \ \text{and} \\
 		(\{z_0\}\times\cplx)\cap\zer(\npee(\bcdot\,; x,y))\subset \{z_0\}\times\Dee,
 		\end{cases} 
\end{multline}
where the polynomial $\npee(\bcdot\;x,y)$ is defined by
\[
 \npee(z,w; x,y)\,:=\,wQ_n(z; y) - P_n(z; x).
\]

\noindent{{\bf Claim.} {\em For any $z_0$, the converse of \eqref{E:z-struc} holds true.}}

\noindent{To see this, let us abbreviate the statement \eqref{E:z-struc} as
$\mathcal{P}(z_0)\impl \mathcal{Q}(z_0)$.  
Now fix a $z_0$ and suppose that it satisfies the three conditions in $\mathcal{Q}(z_0)$.
If $P(z_0; x) = 0$, then by \eqref{E:when-null} $\mathcal{P}(z_0)$ is vacuously true.
Hence, let us assume that $P_n(z_0; x)\neq 0$. Then:
\[
 (\{z_0\}\times\cplx)\cap\zer(\npee(\bcdot\,; x,y))\,=\,\left\{(z_0,
 P_n(z_0; x)/Q_n(z_0; y))\right\}\,\equiv\,\{(z_0, w_0)\}
\]
and, by assumption, $0 < |w_0| < 1$.
Thus $(\{z_0\}\times\cplx)\cap\zer(\pee(\bcdot\,; x,y)) = \{(z_0, 1/w_0)\}
\subset  \{z_0\}\times(\cplx\setminus\cDee)$. This establishes the claim.}
\smallskip

The condition for membership of $(x,y)$ in $\quotn{n}$ can be stated as:
\[
 (x,y)\in \quotn{n}\,\iff\,\text{for each $z\in \cDee$}, \ 
 (\{z\}\times\cplx)\cap\zer(\pee(\bcdot\,; x,y))\subset \{z\}\times(\cplx\setminus\cDee).
\]
In view of \eqref{E:vert-z}, \eqref{E:z-struc} and its converse, and \eqref{E:when-null},
the above statement is rephrased as:
\begin{align}
 (x,y)\in \quotn{n}\,\iff\,&\text{for each $z\in \cDee$}, \notag \\
 &\text{$z$ is not a common zero of $P_n(\bcdot\,; x)$ and $Q_n(\bcdot\,; y)$}, \notag \\
 &\text{$Q_n(z; y)\neq 0$, and $|P_n(z; x)/Q_n(z; y)| < 1$.} \label{E:almost-char}
\end{align}

Finally, we make use the following two facts. First: for any fixed $(x,y)$, the
polynomials $P_n(\bcdot\,; x)$ and $Q_n(\bcdot\,; y)$ have a common zero
if and only if $\res(x,y) = 0$\,---\,see, for instance, \cite{vanderWaerden:Av.I70}.
Second: since $\vP(\bcdot\,; x,y)$ (as defined in
Section~\ref{S:intro}) is a {\bf rational} function,
\[
 \vP(\bcdot\,; x,y)\in \smoo(\cDee)\,\iff\,\text{$\vP(\bcdot\,; x,y)$ is bounded on $\Dee$.}
\]
In view of these two facts, the theorem follows from \eqref{E:almost-char} after an
application of the Maximum Modulus Theorem.
\end{proof}

For our next theorem we shall need the following result by Costara:

\begin{result}[Costara, \cite{costara:osNPp05}, Corollary~3.4]\label{R:costara}
For any $(s_1,\dots, s_n)\in \Cn$, $n\geq 2$, the following assertions are equivalent:
\begin{enumerate}[$(i)$]
 \item The element $(s_1,\dots, s_n)$ belongs to the symmetrized polydisc $\sympd{n}$.
 \item For each $z\in \cDee$, $(\ess_1(z),\dots, \ess_{n-1}(z))\in \sympd{n-1}$, where
 \[
  \ess_j(z)\,:=\,n^{-1}\frac{(n-j)s_j - (j+1)zs_{j+1}}{1 - n^{-1}zs_1}, \ \ j = 1,\dots, n-1.
  \]
\end{enumerate}
\end{result}

\noindent{As in \cite{costara:osNPp05}, implicit in the phrase ``$(s_1,\dots,s_n)\in  \sympd{n}$''
is the sign-convention of the definition:}
\[
\sympd{n}:=\,\left\{(s_1,\dots,s_n)\in\Cn: \text{the roots of} \  
z^n+\sum\nolimits_{j=1}^n(-1)^js_jz^{n-j}=0 \ \text{lie in $\Dee$}\right\}, \
n\in \zahl_+.
\]

\begin{theorem}\label{T:char2}
For any $(x, y)\in \Cn\times \cplx^{n-1}$, $n\geq 3$, the following assertions are
equivalent:
\begin{enumerate}[$(i)$]
 \item The point $(x, y)$ belongs to the $\mu_{1,\,n}$-quotient $\quotn{n}$.
 \item For each $\xi\in \cDee$, the point $(\exx(\xi), \aye(\xi))\in \quotn{n-1}$, where
 \begin{align*}
  \exx_j(\xi)\,&:=\,\frac{(n-j)x_j - j\xi x_{j+1}}{(n-1) - \xi y_1}, \ \ j = 1,\dots, n-1, \\
  \aye_j(\xi)\,&:=\,\frac{(n-1-j)y_j-(j+1)\xi y_{j+1}}{(n-1) - \xi y_1}, \ \ j = 1,\dots, n-2.
 \end{align*}
\end{enumerate}
\end{theorem}
\begin{proof}
Fix an integer $N\geq 2$ (this $N$ is unrelated to the $n$ in the theorem above).
For $(x,y)\in \cplx^N\times\cplx^{N-1}$, let $P_N(\bcdot\,; x)$ and
$Q_N(\bcdot\,; y)$ be as in the proof of the previous theorem.
Note that the following statements are equivalent:
\begin{enumerate}[$a)$]
 \item The point $(x,y)\in \cplx^N\times \cplx^{N-1}$ belongs to $\quotn{N}$.
 \item For each fixed $w\in \cDee$, the zeros of the polynomial $(Q_N(z; y) - wP_N(z; x))$
 lie in $(\cplx\setminus\cDee)$.
 \item For each fixed $w\in \cDee$, the zeros of the polynomial $(Q_N(z; y) - wP_N(z; x))$
 lie in $(\cplx\setminus\cDee)$ and $(1 - wx_1)\neq 0$.
 \item For each fixed $w\in \cDee$, the zeros of the polynomial
 \[
  \frac{z^{N-1}}{1-wx_1}\left(Q_N\left(\frac{1}{z}; y\right) -
  wP_N\left(\frac{1}{z}; x\right)\right)\,=\,z^{N-1} + \sum\noindent_{j=1}^{N-1}(-1)^j
  								    \frac{y_j - wx_{j+1}}{1 - wx_1}z^{N-(j+1)}
 \]
 lie in $\Dee$.
 \item For each fixed $w\in \cDee$,
 \begin{equation}\label{E:quotn-to-sympd}
  \left(\frac{y_1-wx_2}{1 - wx_1},\dots,\frac{y_{N-1}-wx_N}{1 - wx_1}\right)\in \sympd{N-1}.
 \end{equation}
\end{enumerate}
Except, perhaps, for the implication $(b)\!\Rightarrow\!(c)$, it is either self-evident or follows from
definitions that each statement in the above list is equivalent to the one that follows it. 
As for the implication $(b)\!\Rightarrow\!(c)$: it follows from $(b)$ that if the polynomial in $(b)$ is
nonconstant, then the product of its zeros must be non zero, and if it is constant (for a
fixed $w\in \cDee$), then this constant must be non-zero. In either case, this gives
$(1-wx_1)\neq 0$.
\smallskip

Now consider $n\geq 3$ as given.
From the equivalence $(a)\!\iff\!(e)$ with $N = n$, and from Costara's theorem, we get:
\begin{itemize}
 \item[$(\blacktriangle)$] The point $(x,y)\in \quotn{n}$$\iff$for each $(w,\xi)\in (\cDee)^2$,
 $(\ess_1(\xi, w; x,y),\dots,\ess_{n-2}(\xi, w; x,y))$ belongs to $\sympd{n-2}$, where
 \begin{multline*}
  \ess_j(\xi, w; x,y)\,:=\,\frac{\left(\dfrac{n-j-1}{n-1}\right)\left(\dfrac{y_j - wx_{j+1}}{1 - wx_1}\right) - 
 						 \xi\,\left(\dfrac{j+1}{n-1}\right)\left(\dfrac{y_{j+1} - wx_{j+2}}{1 - wx_1}\right)}
 						 {1-\left(\dfrac{\xi}{n-1}\right)\left(\dfrac{y_1-wx_2}{1 - wx_1}\right)}, \\
  j = 1,\dots,n-2.
 \end{multline*}
\end{itemize}
Observe that the expressions for $\ess_j(\xi, w; x,y)$ can be rewritten as
\begin{multline}\label{E:esss}
 \ess_j(\xi, w; x,y)\,:=\,\frac{\left[\dfrac{(n\!-\!1\!-\!j)y_j-(j\!+\!1)\xi y_{j+1}}{(n\!-\!1) - \xi y_1}\right] -
 						w\left[\dfrac{(n\!-\!j\!-\!1)x_{j+1} - (j\!+\!1)\xi x_{j+2}}{(n-1) - \xi y_1}\right]}
 						{1 - w\left[\dfrac{(n\!-\!1)x_{1} - \xi x_{2}}{(n\!-\!1) - \xi y_1}\right]}, \\
 j = 1,\dots,n-2.
\end{multline} 

For $N\geq 2$, it follows from the equivalence $(a)\!\iff\!(e)$ that we established
above that (just take $w = 0$ in \eqref{E:quotn-to-sympd})
\[
 (x,y)\in \quotn{N}\,\impl\,(y_1,\dots,y_{N-1})\in \sympd{N-1}\,\impl\,|y_1| < N-1.
\]
From this, we get
\begin{equation}\label{E:auxil}
 (x,y)\in \quotn{N}\,\iff\,\big((x,y)\in \quotn{N} \ \; \text{and, for each 
 $w\in \cDee$, $(1 - wy_1)\neq 0$}\big).
\end{equation}

We now apply the equivalence $(a)\!\iff\!(e)$ taking $N = (n-1)$ (which is valid,
since, by hypothesis, $(n-1)\geq 2$). From \eqref{E:auxil}, the equivalence $(\blacktriangle)$, and by
comparing \eqref{E:esss} with \eqref{E:quotn-to-sympd}, we see that 
for any $(x,y)\in \Cn\times\cplx^{n-1}$, each assertion
in the list below is equivalent to the one following it:
\begin{enumerate}[A)]
 \item The point $(x,y)\in \Cn\times \cplx^{n-1}$ belongs to $\quotn{n}$.
 \item The point $(x,y)\in \Cn\times \cplx^{n-1}$ belongs to $\quotn{n}$ and,
 for each $w\in \cDee$, $(1 - wy_1)\neq 0$.
 \item For each $w\in \cDee$, $(1 - wy_1)\neq 0$ and, for each $\xi\in \cDee$,
 $(\ess_1(\xi, w; x,y),\dots,\ess_{n-2}(\xi, w; x,y))$ belongs to $\sympd{n-2}$, where
 $\ess_j(\xi, w; x,y)$, $j = 1, \dots, (n-2)$, is given by $(\blacktriangle)$.
 \item The assertion $(ii)$ in the statement of Theorem~\ref{T:char2}.
\end{enumerate}
This completes the proof.  
\end{proof}
\smallskip

\begin{remark}\label{Rem:DOM}
In Section~\ref{S:intro}, we mentioned that the sets $\quotn{n}$ are domains. That
each $\quotn{n}$, $n\geq 2$, is open can be established by a classical argument. It can be deduced
from the fact that the condition defining $\quotn{n}$ is an open condition; that $\cplx$-affine algebraic
hypersurfaces of degree\,$\leq n$ vary continuously\,---\,in an appropriate sense; see
\cite[Chapter~1, \S\,1.2]{chirka:cas1989}\,---\,with respect to the coefficients of their defining
functions; and that the varieties occurring in the definition of $\quotn{n}$ have a rather simple
form. {\bf However}, the proof of the previous theorem provides a slick way of establishing the
openness of $\quotn{n}$. Fix an $n\geq 2$ and let $(x_0, y_0)\in \quotn{n}$. By the implication
$(a)\Rightarrow (e)$ in the above proof, we get, for each $w\in \cDee$:
\[
 (Z^w_1,\dots, Z^w_{n-1})\,:=\,\left(\frac{y_{0, 1}-wx_{0, 2}}{1 - wx_{0, 1}},\dots,
 \frac{y_{0,\,n-1}-wx_{0,\,n}}{1 - wx_{0, 1}}\right)\in \sympd{n-1}.
\]
As $\sympd{n-1}$ is open and $w$ varies through a compact set,
there exists an $\eps > 0$ such that
the polydiscs $\boldsymbol{\Delta}_w :=D(Z^w_1; \eps)\times\dots\times D(Z^w_{n-1}; \eps)
\subset \sympd{n-1}$. By the implication $(a)\Rightarrow (c)$, the
set $(-x_{0, 1}\cDee + 1)\not\ni 0$, whence
by equicontinuity we can find a $\delta > 0$ such that
\[
 \left(\frac{y_1-wx_2}{1 - wx_1},\dots,\frac{y_{n-1}-wx_n}{1 - wx_1}\right)\in \boldsymbol{\Delta}_w \ \ \
 \forall (x, y)\in \bol^{2n-1}((x_0, y_0); \delta)
\]
(where $\bol^N(a; r)$ denotes the open Euclidean ball centered at $a\in \cplx^N$ of radius $r$), for
{\em each} $w\in \cDee$.
This time, by the implication $(e)\Rightarrow (a)$, we get
$\bol^{2n-1}((x_0, y_0); \delta)\subset \quotn{n}$. It follows that $\quotn{n}$ is open. The
connectedness of $\quotn{n}$ is a consequence of part~(2) of Theorem~\ref{T:equiv}.    
\end{remark}

\section{Proofs of the Main Theorems}\label{S:main}

The results of the last two sections provide us all the tools needed to prove
Theorems~\ref{T:equiv} and \ref{T:nec}.

\begin{proof}[The proof of Theorem~\ref{T:equiv}]
We begin by reminding the reader of the notational comment in Remark~\ref{Rem:column}.
Recall further: if $S$ is a square matrix, then $\cmpn{S}$ will denote the companion matrix
of its characteristic polynomial (normalized as in Definition~\ref{D:nonderog}).
\smallskip

\noindent{{\bf 1)} Let $\Lambda$ denote
the holomorphic identification $\Lambda: \cplx^{n\times n}\lrarw
\cplx^{(n-1)\times (n-1)}\times\cplx^{n-1}\times\Cn$,
\[
 \Lambda(A)\,:=\,\big(\de{A},\,(a_{j,1})_{2\leq j\leq n},\,(a_{1,\,k})_{1\leq k\leq n}\big),
\] 
writing $A = [a_{j,\,k}]$. Define
\begin{align*}
 \diff\,&:=\,\{X\in \cplx^{(n-1)\times(n-1)} : X \ \text{is non-derogatory}\}, \\
 \mathfrak{S}^1\,&:=\,(\cplx^{(n-1)\times(n-1)}\setminus\diff)\times\cplx^{n-1}\times\Cn.
\end{align*}}
Define the function
$\Theta: \diff\times\cplx^{n-1}\lrarw \cplx$ as follows:
\begin{equation}\label{E:func}
 \Theta(X, v)\,:=\,\det\big(\,[ v \ \ Xv\,\dots\,X^{n-2}v]\,\big).
\end{equation}
Fix some $X^0\in \diff$. As $X^0$ is non-derogatory, it has a cyclic vector: call it 
$\cyc{X^0}$. Clearly, $\Theta(X^0, \cyc{X^0}) \neq 0$, whence 
$\Theta(X^0;\,\bcdot)\not\equiv 0$, and this is true for any $X^0\in \diff$. By construction,
$\Theta(X;\,\bcdot)$ and $\Theta$ are holomorphic functions.
Since $\Theta(X;\,\bcdot) \not\equiv 0$ (for $X\in \diff$) and $\Theta\not\equiv 0$, it is a classical
result\,---\,see, for instance, \cite[Theorem~14.4.9]{rudin1980:FT}\,---\,that
\begin{align}
 \Theta(X;\,\bcdot)^{-1}\{0\}\times\Cn \varsubsetneq &\,(\{X\}\times\cplx^{n-1}\times\Cn) \ 
 \text{has zero ($(4n-2)$-dim'l.)} \ \notag \\
 &\,\text{Lebesgue measure (for each $X\in \diff$),} \notag \\ 
 \Theta^{-1}\{0\}\times\Cn \varsubsetneq &\,(\diff\times\cplx^{n-1}\times\Cn) \ \text{has
 zero ($2n^2$-dim'l.) Lebesgue measure.} \label{E:lebMeas}
\end{align}
Note that, for a matrix $X\in \diff$, $\Theta(X, v)\neq 0\,\iff\,v \ \text{is a cyclic vector of $X$}$.
Hence, writing $\Theta^{-1}\{0\}\times\Cn =:\mathfrak{S}^2$, we get 
\begin{equation}\label{E:geneLoc}
 \gene\,=\,\OM_{1,\,n}\!\cap \Lambda^{-1}((\diff\times\cplx^{n-1}\times\Cn)\setminus\mathfrak{S}^2).
\end{equation}
Since $\Lambda^{-1}(\mathfrak{S}^1)$ has zero ($2n^2$-dimensional) Lebesgue measure, it follows
from \eqref{E:lebMeas} and \eqref{E:geneLoc} that $(\OM_{1,\,n}\setminus\gene)$ has zero
Lebesgue measure.
\medskip

\noindent{{\bf 2)} Part~(2) is essentially the first part of Proposition~\ref{P:mu-to-quot}. 
That $\pi_n$ is holomorphic is trivial as it is a polynomial map.}
\medskip

\noindent{{\bf 3)} If there exists a holomorphic map $F : \Dee\lrarw \OM_{1,\,n}$ that interpolates the
given data, then, by part~(2), $f := \pi_n\circ F$ has the required properties.}
\smallskip

Let us now assume that there exists a holomorphic map $f : \Dee\lrarw \quotn{n}$ such that
$f(\zt_j) = \pi_n(W_j)$ for every $j$. Let us write $f = (\bsf{x}, \bsf{y})$, where
$\bsf{x} =: ({\sf x}_1,\dots, {\sf x}_n) : \Dee\lrarw \Cn$ and
$\bsf{y} =: ({\sf y}_1,\dots, {\sf y}_{n-1}) : \Dee\lrarw \cplx^{n-1}$. Let
$p_1,\dots p_{n-1}$ be the polynomials given by Lemma~\ref{L:imp}. Define the
holomorphic map $\phi : \Dee\lrarw \cplx^{n\times n}$ as follows:
\begin{equation}\label{E:1st_approx}
 \phi(\zt)\,:=\,\begin{bmatrix}
 			 \ {\sf x}_1(\zt) &\vline & p_1\circ f(\zt)	& \dots & p_{n-2}\circ f(\zt)	& p_{n-1}\circ f(\zt) \ \\ \hline
 			 \ 1                  &\vline & 0                     & {}    & {}			& (-1)^n{\sf y}_{n-1}(\zt) \ \\
 			 \ 0                  &\vline & 1                     & 0     & \text{\LARGE{$\boldsymbol{0}$}}
 			 & (-1)^{n-1}{\sf y}_{n-2}(\zt) \ \\
 			 \ \vdots           &\vline & {}                    & \ddots & \ddots		& \vdots \ \\
 			 \ 0                  &\vline & \text{\LARGE{$\boldsymbol{0}$}}         & {}  & 1 &  {\sf y}_{1}(\zt) \
 			 \end{bmatrix}.
\end{equation}
Note that, in the notation of Lemma~\ref{L:imp}, $\phi = B(\bsf{x}, \bsf{y})$. Hence, it follows from
the second assertion in Proposition~\ref{P:mu-to-quot} that $\phi : \Dee\lrarw \OM_{1,\,n}$. And
it follows from Lemma~\ref{L:imp} that
\begin{equation}\label{E:interp1}
 \pi_n\circ\phi(\zt_j)\,=\,\pi_n(W_j), \ \ j = 1,\dots, M.
 \end{equation}
\vskip1mm

The above $\phi$ is not, in general, the desired $F$ (although the range
of $\phi$ {\em is} contained in $\OM_{1,\,n}$). We must now address this problem.
Let $\cplx^*\oplus GL_{n-1}(\cplx)$ be as introduced just before the statement of
Lemma~\ref{L:orbit}. The importance of this group to our discussion is the following simple
(but powerful):
\smallskip

\noindent{{\bf Fact.} {\em For a matrix $A\in \cplx^{n\times n}$,
$\mu_{E^{1,\,n}}(A) =
\mu_{E^{1,\,n}}(G^{-1}AG)$ for each $G\in \cplx^*\oplus GL_{n-1}(\cplx)$.}}
\medskip
 
\noindent{So, the idea behind what follows is to construct an appropriate holomorphic
$(\cplx^*\oplus GL_{n-1}(\cplx))$-valued map $\psi$, defined on $\Dee$, such that
$F\,:=\psi^{-1}\phi\,\psi$ is the desired interpolant.}
\smallskip

To this end, we point out that by \eqref{E:interp1} and by the definition of the map $\pi_n$, we get 
\begin{equation}\label{E:W-dots}
 \de{\phi(\zt_j)}\,=\,\cmpn{\de{W_j}},  \ \ j = 1,\dots, M.
\end{equation}
Now refer to the proof of Lemma~\ref{L:uni_det}. By the fact that
$W_1,\dots, W_M\in \gene$, there exists a unique
matrix $\Gamma_j\in GL_{n-1}(\cplx)$, $j = 1,\dots, M$, such that
\begin{equation}\label{E:interp2}
 \Gamma_j^{-1}\de{W_j}\Gamma_j\,=\,\cmpn{\de{W_j}},  \ \quad \
 \Gamma_j\,[1 \ 0\,\dots\,0]^T=\,[\,{}^jw_{2,1} \ \ {}^jw_{3,1}\,\dots\,{}^jw_{n,1}]^T, \ \
 j = 1,\dots, M,
\end{equation}
where we write $W_j = [\,{}^jw_{i,\,k}]$ for each $j\leq M$. At this point, we know
two things:
\begin{itemize}
 \item by examining \eqref{E:1st_approx},
 $(1\oplus\Gamma_j)\phi(\zt_j)(1\oplus\Gamma_j)^{-1}$ belongs to $\geep$;
 \item by the above observation, \eqref{E:W-dots}, and \eqref{E:interp2}, if we set 
 $A := (1\oplus\Gamma_j)\phi(\zt_j)(1\oplus\Gamma_j)^{-1}$ and 
 $B := W_j$, then this choice of $(A, B)$ satisfies the hypothesis of Lemma~\ref{L:uni_det};
\end{itemize}
for each $j = 1,\dots, M$.  Here, we have used the abbreviated notation,
introduced in Section~\ref{S:prelim}, for an element in $\cplx^*\oplus GL_{n-1}(\cplx)$.
Therefore, Lemma~\ref{L:uni_det} tells us:
\begin{equation}\label{E:top-row}
 \text{The first row of $(1\oplus\Gamma_j)\phi(\zt_j)(1\oplus\Gamma_j)^{-1}$ equals
 the first row of $W_j$ for each $j\leq M$.}
\end{equation}
 
As each $\Gamma_j$ above is an invertible matrix, there exists a matrix
$L_j\in \cplx^{(n-1)\times(n-1)}$ such that $\exp(L_j) = \Gamma_j$. Let
$\Psi : \Dee\lrarw \cplx^{(n-1)\times(n-1)}$ be any matrix-valued holomorphic
function such that $\Psi(\zt_j) = L_j$, $j = 1,\dots M$. Now, let us
define the following $(\cplx^*\oplus GL_{n-1}(\cplx))$-valued holomorphic map:
\[
 \psi(\zt)\,:=\,1\oplus e^{-\Psi(\zt)} \ \ \forall \zt\in \Dee.
\]
Since we have shown that $\phi(\zt)\in \OM_{1,\,n} \ \forall \zt\in \Dee$, it follows
from the Fact stated above that:
\[
 \psi(\zt)^{-1}\phi(\zt)\psi(\zt)\,\in\,\OM_{1,\,n} \ \ \forall \zt\in \Dee.
\]
We now write $F := \psi^{-1}\phi\psi$. We have just argued that 
$F : \Dee\lrarw \OM_{1,\,n}$ and is holomorphic. From \eqref{E:1st_approx}, 
\eqref{E:W-dots}, \eqref{E:interp2}, and \eqref{E:top-row}, it follows that this $F$ is the
desired interpolant. 
\end{proof}

The ideas used in the above proof lead to some observations that would be relevant
when dealing with the unit ``$\mu_E$-balls'' when $E$ is of greater complexity.

\begin{remark}\label{Rem:good-q}
Probably the most important role in the proof of Theorem~\ref{T:equiv} was played
by the fact that the group $\cplx^*\oplus GL_{n-1}(\cplx)$ acts on $\OM_{1,\,n}$.
A more abstract look into the relationship between this action and the domain 
$\quotn{n}$ might suggest the way forward in formulating analogues of
Theorem~\ref{T:equiv} for more general cases of $E$. For both the pairs $(\OM_n, \sympd{n})$ and
$(\OM_{1,\,n}, \quotn{n})$, $n\geq 2$, it turns out that the relationship of the
lower-dimensional domain to its associated unit ``$\mu_E$-ball'' is analogous
to the categorical quotient associated to an affine algebraic variety with a reductive
group acting on it. We say ``analogous'' because $\OM_{1,\,n}$ is not an
algebraic variety. But there {\em are} settings\,---\,see \cite{snow:rgaSs82} by
Snow, for instance\,---\,to which the constructions of classical geometric invariant
theory carry over. In this work, owing to the nature
of the ``structural space'' $E^{1,\,n}$, we did not need to appeal to the abstract
theory (which still needs some enhancements to Snow's work). However, in
that language, the components of the map $\pi_n$ are the generators of
the ring of $G$-invariant functions, $\quotn{n}$ is the analogue of the
categorical quotient, and $\gene$ is the union of all {\bf closed} $G$-orbits of
$\OM_{1,\,n}$, $G = \cplx^*\oplus GL_{n-1}(\cplx)$.
For a general $E$, Lemma~\ref{L:det} will {\em still} give us the generators of
the ring of $G$-invariant functions on $\OM_E$ (for an appropriate $G$).
However, when $E$ is of much greater complexity, the abstract viewpoint
hinted at might be helpful in determining
the analogue of the set $\gene$ without engaging in ever more
complex computations.
\end{remark} 

We now come to the proof of Theorem~\ref{T:nec}. This proof is an 
 easy consequence of part~(2) of the previous theorem and Theorem~\ref{T:char1}.
\vspace{-1mm}

\begin{proof}[The proof of Theorem~\ref{T:nec}]
In view of Theorem~\ref{T:equiv}-$(2)$, the function $\pi_n\circ F$ is a holomorphic map
and $\pi_n\circ F(\Dee) \subset \quotn{n}$. Thus, for each
$\zt\in \Dee$, $\pi_n\circ F(\zt)\in \quotn{n}$. It follows from this, from
condition~(I) in Theorem~\ref{T:char1}, 
and the Maximum Modulus Theorem that if we {\em fix} a point $z\in \cDee$, then
\begin{equation}\label{E:bound}
 |\vP(z\,;\,X\circ F(\zt), Y\circ F(\zt))|\,<\,1 \ \ \text{for each $\zt\in \Dee$}.
\end{equation}
It is obvious that the functions
\[
 \zt\!\longmapsto P_n(z\,;\,X\circ F(\zt)), \quad \quad
 \zt\!\longmapsto Q_n(z\,;\,Y\circ F(\zt)),  \ \ \zt\in \Dee, 
\]
are holomorphic functions. Thus, it follows from the bound \eqref{E:bound} that
$f^z$ defined by
\[
 f^{z}(\zt)\,:=\,\vP(z\,; X\circ F(\zt), Y\circ F(\zt)), \ \ \zt\in \Dee,
\]
is a holomorphic $\Dee$-valued function. This function, for each fixed $z\in \cDee$, satisfies
\[
 f^z(\zt_j)\,=\,\vP(z\,;\,X(W_j), Y(W_j)), \ \ j = 1,\dots, M.
\]
It follows from the classical result by Pick that the $M\times M$ matrix $M_z$ is positive
semi-definite.
\end{proof}
\vspace{-2mm}

We end this article with an observation:

\begin{remark}\label{Rem:returns}
Two major effects of the idea introduced by Agler--Young in \cite{aglerYoung:cldC2si99}\,---\,of
which this work is an extension\,---\,are the reduction in the dimensional complexity of the
problem $(*)$, and the ability to deduce necessary conditions for Nevanlinna--Pick interpolation
such as Theorem~\ref{T:nec}. However, the discussion in Remark~\ref{Rem:good-q} suggests that
the advantage to be gained from the first of those two features has certain limits.
As the number of the disparate diagonal blocks determining
$E$ increases, the number of generators of the ring of $G$-invariant functions on
$\OM_E$ (for an appropriate reductive
group $G$ naturally associated with $\OM_E$ and acting on it by conjugation) would tend
to grow; see Lemma~\ref{L:det} above. This implies that there would be diminishing advantage, in
terms of reduction in dimensional complexity of the problem $(*)$, in working with
analogues of $\sympd{n}$ or $\quotn{n}$. 
\end{remark}

\end{document}